\documentclass{amsart}

\usepackage[T1]{fontenc}

\usepackage{amsmath,amssymb,amsthm,amscd,comment,cite,amsfonts,indentfirst,color,setspace,bbold}
\usepackage[mathscr]{euscript}
\usepackage{mathtext,marvosym,textcomp}
\usepackage[makeroom]{cancel}
\usepackage{array}
\usepackage{mathtools}
\usepackage{hyperref}
\hypersetup{colorlinks,allcolors=black}

\newcommand{\hm}{\mathscr{H}}
\newcommand\numberthis{\addtocounter{equation}{1}\tag{\theequation}}
\newcommand{\ts}{\tilde{\Sigma}}
\newcommand{\s}{\Sigma}
\newcommand{\gh}{\overline{\mathcal{H}(g)}}

\newtheorem{theorem}{Theorem}
\newtheorem{corr}{Corollary}
\newtheorem{proposition}{Proposition}
\newtheorem*{remark}{Remark}
\newtheorem{lemma}{Lemma}
\newtheorem{conjecture}{Conjecture}
\newtheorem{definition}{Definition}

\newtheoremstyle{named}{}{}{\itshape}{}{\bfseries}{.}{.5em}{\thmnote{#3}}
\theoremstyle{named}
\newtheorem*{namedprop}{Proposition}

\title{Proof of the bounded conformal conjecture}
\author{Sameer Kumar}
\address{Moscow Institute of Physics and Technology, Institutskiy pereulok 9, Dolgoprudny, 141700, Russia}
\email{kumar.samip@phystech.edu}
\date{}

\begin{document}

\begin{abstract}
    Given any asymptotically flat 3-manifold $(M,g)$ with smooth, non-empty, compact boundary $\Sigma$, the conformal conjecture states that for every $\delta>0$, there exists a metric $g' = u^4 g$, with $u$ a harmonic function, such that the area of outermost minimal area enclosure $\ts_{g'}$ of $\Sigma$ with respect to $g'$ is less than $\delta$. Recently, the conjecture was used to prove the Riemannian Penrose inequality for black holes with zero horizon area, and was proven to be true under the assumption of existence of only a finite number of minimal area enclosures of boundary $\Sigma$, and boundedness of harmonic function $u$. We prove the conjecture assuming only the boundedness of $u$.
\end{abstract}

\maketitle

\section{Introduction}

The definition of mass of objects in general relativity has been a fascinating topic for both mathematicians and physicists alike since decades. The ADM formalism \cite{adm}, named after its authors and describing the mass of an asymptotically flat manifold $(M,g)$,  is a widely studied topic in mathematical general relativity. It describes the mass of an asymptotically flat end of a manifold $(M,g)$ as a flux integral of the gravitational field over a large celestial sphere. A end of manifold $(M,g)$ is said be asymptotically flat if it is diffeomorphic to $\mathbb{R}^3 \backslash K$, where $K$ is a compact set, and the metric $g_{ij}$ satisfies the following conditions \cite{huiskenilmanen}: 

\begin{equation*}
    |g_{ij} - \delta_{ij}| \leq \frac{C}{|x|}, \;\; |g_{ij;k}| \leq \frac{C}{|x|^2}, \;\; Rc \geq - \frac{Cg}{|x|^2}
\end{equation*}

as $|x| \rightarrow \infty$, and where the derivatives in $g_{ij;k}$ are taken with respect to the flat metric on $\mathbb{R}^3 \backslash K$.

The ADM formalism defines the mass of an asymptotically flat spacetime $(M,g)$ as -

\begin{equation} \nonumber
    m_{ADM} = \lim_{r \rightarrow \infty} \frac{1}{16\pi}\int\limits_{\partial B_r} (g_{ij;i} - g_{ii;j})\nu^{j} d\mu_{\delta}
\end{equation}

where $\nu^j$ is the normal vector in j-th direction, and $d\mu_{\delta}$ represents the flat metric. \\

By Stokes' theorem, the above expression is simply the flux integral of the field form generated by the value under the integral through a large celestial sphere. This is similar to ``description of charge" in classical electrodynamics, and we can call the mass of an object as the \textit{``gravitational charge"} generating its gravitational field. \\

In this paper, by a \textit{black hole} we mean a connected component of the outermost compact minimal surface in $(M,g)$. We will give a more precise description of outermost minimal surfaces in the next section. Penrose \cite{Penrose:1969pc} conjectured that the total mass contributed by a black hole should be at least $\sqrt{A/16\pi}$, where $A$ is the area of the black hole. Penrose made this argument using physical principles; showing that if the condition was not satisfied, the manifold would admit Cauchy data for the Einstein's equation that would lead to naked singularities - a contradiction to the so called ``cosmic censorship principle", which states that singularities are always encapsulated by a black hole. This inequality is known as the \textit{Penrose inequality}. Schoen and Yau \cite{schoenyau1, schoenyau2} proved the following result known as the \textit{Positive mass theorem}.\\

\begin{theorem} (Positive mass theorem)
    Let $(M,g)$ be an asymptotically flat Riemannian 3-manifold with nonnegative scalar curvature. Then $m_{ADM} \geq 0$, with equality in the case when M is isometric to $\mathbb{R}^3$.
\end{theorem}

Although less sharper than originally conjectured by Penrose, the paper was the first prominent step in the direction of the Penrose inequality. In \cite{witten1981new} another proof of was presented by Witten, with techniques involving spinors. The paper also lay forward the idea of a connection of  supergravity-like compactifications occuring due to possible instability of Minkowski space in higher dimensions, due to a possibly negative gravitational (ADM) mass. While it was unknown at that time, it is now known that the positive mass theorem holds for all dimensions \cite{schoen2017positive}, rendering Minkowski space as a stable vaccum solution to Einstein equations in all dimensions. Following the developments of Shoen-Yau and Witten, Huisken and Ilmanen \cite{huiskenilmanen} proved a sharper bound by considering the area of only a single connected component of $\Sigma = \partial M$ (equivalent to considering the ``largest blackhole" in the manifold) using inverse mean curvature flow. Almost simultaneously, Bray \cite{bray1999proof} proved the general case of multiple connected components of boundary $\partial M$ (i.e., multiple black hole horizons) using the positive mass theorem. \\ 

\begin{theorem}(Riemannian Penrose inequality)
    Let $(M,g)$ be an asymptotically flat 3-manifold with nonnegative scalar curvature such that $\partial M$ is smooth, compact and non empty. Then $m_{ADM} \geq \sqrt{\frac{A}{16\pi}}$, with the equality taking place when $(M,g)$ is isometric to Shwarzchild manifold with mass $m_{ADM}$. 
\end{theorem}

Here, $A$ denotes the area of a single connected component of $\partial M$ in the theorem of Huisken and Ilmanen \cite{huiskenilmanen}, and the area of $\partial M$ in the theorem of Bray \cite{bray1999proof}. For a thorough review of the above-said developments, see \cite{Bray:2003ns}.

The current work is connected to a progression of the work stated above, for manifolds where the metric has a so-called \textit{zero area singularity}. Zero area singularities (ZAS) are horizons (i.e. components of the boundary $\partial M$) where the metric trivializes, and hence the area of the horizon tends to zero in a weak surface convergence sense. This is the case, for example, of a Shwarzchild metric with a black hole with ``negative mass". Mass bounds for zero area singularities concerns the stability of the existence of such black holes. We will discuss the technical details regarding the convergence of surfaces, and the definition of ZAS in section \ref{section: formalisms}. \\

The conformal conjecture was introduced by Bray in \cite{Bray:2009voj} to prove the \textit{Riemannian ZAS inequality}. Jauregui \cite{jeff} introduced a weaker statement of the above conjecture, equivalent to the original. 

\begin{conjecture} \label{conjecture:conf_conj} (Conformal conjecture)
    Given $\delta > 0$, there exists $g' \in \overline{\mathcal{H}(g)}$ with outermost minimal area enclosure $\ts_{g'}$ such that : 
    \begin{enumerate}
        \item $\ts_{g'}$ is disjoint from $\Sigma$, and
        \item $| |\ts_{g'}|_{g'} - |\Sigma|_{g'} | < \delta$ 
    \end{enumerate}
\end{conjecture}

The above statement is readily verified to hold in the simplest case: the spherically symmetric metric: all minimal area enclosures in this case are spheres that are either equal to the boundary or completely disjoint from it (and hence a minimal surface since smoothness is guaranteed by the results of Huisken-Ilmanen). The harmonic function in this case can be picked to be a scalar, say $c$, on the boundary and approaching $1$ at infinity. If $c$ is taken to be very small then the minimal area enclosure would be the boundary itself, while taking $c$ very large would explode the metric near the boundary and hence make its minimal area enclosure disjoint from it. Picking some intermediate value of $c$ would result in a neck-like structure as we move away from the boundary and towards infinity. Hence, we get a minimal area enclosure that is disjoint from the boundary, and we can also make the areas of the two surfaces arbitrarily close by carefully picking the parameter $c$. Our goal is to estabilish this result for an arbitrary Riemannian metric $g$. In this paper, we prove the following theorem. 

\begin{theorem} \label{theorem: main_theorem} (Main theorem)
    Conformal conjecture is true under the following additional assumption:
    \begin{enumerate}
        \item the maximizer $g' = u^4g$ of Theorem 3 below satisfies $u \leq C$ for a constant $C > 0$.
    \end{enumerate}
\end{theorem}

Here, $\overline{\mathcal{H}(g)}$ is the generalised harmonic conformal class (discussed further in section \ref{section: formalisms}). The above theorem is an improvement of the following previously known result, and our approach is a generalisation of the approach developed in \cite{jeff}.

\begin{theorem} \label{theorem: maintheorem_jau} (Theorem 29 \cite{jeff})
    Conformal conjecture is true under the following additional assumptions:
    \begin{enumerate}
        \item the maximizer $g' = u^4g$ of Theorem 3 below satisfies $u \leq C$ for a constant $C > 0$, and
        \item the number of minimal area enclosures of $\Sigma$ with respect to $g'$ is finite.
    \end{enumerate}
\end{theorem}

This paper is structured as follows: section \ref{section: formalisms} contains the background concepts to be used later, section \ref{section: heuristic_proof} contains a heuristic proof of the main theorem, section \ref{section: proof} contains the preliminary lemmas leading to the proof of the main theorem, section \ref{section: main_theorem} contains our proof of the main theorem, and section \ref{section: discussion} contains a short discussion on some open problems related to this paper, and possible approaches for solving them.

\section*{Acknowledgement}

I wish to thank Prof. Jeffrey L. Jauregui for the helpful and motivating correspondence, and many useful remarks that helped in the direction of this research. Their doctoral thesis is the central motivation for this paper. I also wish to thank my advisor Dr. Edvard Musaev for constant support and encouragement, and Rados for unwavering support and help with the correction of typos.


\section{Surfaces, the harmonic class, and density} \label{section: formalisms}


The constructions in this section largely follow \cite{jeff}, around which our argument has been constructed.

\subsection{Surfaces and minimal area enclosures} \label{section: surfaces and mae}

Surfaces can be defined in many ways in geometry. One useful way to define surfaces is as functionals acting on forms over them, as followed in \cite{jeff} (see also \cite{federer1969geometric}, \cite{simon}). We call this object a current. 

\begin{definition}
    (Currents) Let M be a manifold of dimension $m \geq n$. An n-current in $M$ is a linear functional on the space $C_{o}^{\infty}(\Lambda^n M)$ of compactly supported, smooth, differential n-forms. 
\end{definition}

Submanifolds define current in a natural way - let $U \subset M$ be a k-submanifold ($ k \geq n$) of M. Then an n-current associated to $U$ is defined as 

\begin{equation} \nonumber 
    U(\phi) = \int\limits_{U} \phi
\end{equation}

for $\phi \in C_{o}^{\infty}(\Lambda^n U)$. \\

A \textit{boundary current} is made to conform to Stokes' theorem: $\partial S(\phi) = S(\partial \phi)$. An \textit{integral n-current} is a current whose boundary is an n-1 current. We can also define the support of a current in a natural way. 

\begin{definition}

    The support of an n-current $S$ (supp $\phi$) is defined as the intersection of all closed sets $C$ such that 

    \begin{equation*}
        \text{supp } \phi \cap C = \emptyset \Rightarrow S(\phi) = 0
    \end{equation*}

    for all $\phi \in C_{o}^{\infty}(\Lambda^n S)$, where $\text{supp } \phi$ is the support of $\phi$.
    
\end{definition}

Next, we define the \textit{mass norm} of a current $S$ in $M$.

\begin{definition}
    The mass norm $|S|$ of an n-current $S \subset M$ is defined as 
    \begin{equation*}
    |S| = \sup\limits_{\phi} \{ S(\phi), \; \phi \in C_{o}^{\infty}(\Lambda^n S): ||\phi|| \leq 1 \}
    \end{equation*}

    where $|| \cdot ||$ is the max norm for differential n-form $\phi$.

\end{definition}

A surface in $M$ is hence realised as a 2-current in $M$. We can also define the convergence of surfaces in terms of currents as follows. 

\begin{definition} 
    A sequence of n-currents $\{S_i\}$ is said to converge (weakly) to an n-current $S$ if 

    \begin{equation*}
        S_i (\phi) \rightarrow S(\phi)
    \end{equation*}
    for all $\phi \in C_{o}^{\infty}(\Lambda^n \mathbb{R}^{n+k})$, where $k \geq 0$.
\end{definition}

The above definitions get a more geometric meaning once we define a particular choice of measure in $\mathbb{R}^{n+k}$.

\begin{definition}
    The Hausdorff measure of a set $S \subset \mathbb{R}^{n+k}$ is defined as:

    \begin{equation*}
        \hm^{n}(S) = \lim\limits_{\delta \rightarrow 0} \inf\limits_{\{C_j\}} \sum\limits_{j=1}^{\infty} \omega_n \Big( \frac{\text{diam }C_j}{2} \Big) 
    \end{equation*}
    where $\omega_n$ is the volume of a unit ball in $\mathbb{R}^n$, and the infimum is taken over all countable collections $\{C_j\}$ covering $S$, such that $\text{diam }C_j < \delta$ for each j.
\end{definition}

Then the current of a $\hm^n$ measurable set $\Omega \subset \mathbb{R}^{n}$ can be defined as a functional acting on $\phi \in C_{o}^{\infty}(\Lambda^n \mathbb{R}^{n})$ 

\begin{equation*}
    S(\phi) = \int\limits_{\Omega} \rho d\hm^n
\end{equation*}

since any $\phi \in C_{o}^{\infty}(\Lambda^n \mathbb{R}^{n+k})$ can be decomposed as $\phi = \rho \omega$, with $\omega$ being the volume form for $\mathbb{R}^n$. This, in particular, allows us to equate the mass norm to the Hausdorff measure of the set: $|S| = \hm^n(S)$. \\

We denote the Hausdorff k-measure on $M$ with respect to the metric $g$ as $\hm^k_g$. Let $\Omega \subset M$ be a $\hm^3_g$-measurable set. We will not distinguish notationally between $\Omega$ being viewed as a set or a current and use the symbol $\Omega$ for both purposes. We let $\Sigma = -\partial M$ (the normal at the boundary points into the manifold). We can now define an enclosing surface as follows.

\begin{definition}
    Let $\Omega \subset M$ be an $\hm^3_g$-measurable, compact set that defines an integral 3-current (of multiplicity one). A surface enclosing $\Sigma$ is a 2-current S of the form

    \begin{equation*}
        S = \Sigma + \partial \Omega
    \end{equation*}

    We say that $S$ encloses $\Sigma$ properly if $S$ does not intersect $\Sigma$. Also, we say that $S$ is $C^{k,a}$ if the supp $S$ is $C^{k,a}$ embedded submanifold of M. If $S_i = \Sigma + \partial \Omega_i$, $i = 1,2$, then we further say $S_1$ encloses $S_2$ if supp $S_2 \subset$ supp $S_1$.
\end{definition}

By the statement $p \in S$, we mean that $p \in \text{supp } S$, $S \cap \Sigma$ means the restriction of current $S$ to the set $\Sigma$ (defined in the integral sense as the integral being evaluated on the intersection $S \cap \Sigma$, also denoted as $S \llcorner \Sigma$), and $S \backslash \Sigma$ as the restriction of $S$ to the set $M \backslash \Sigma$.

\begin{definition}
    A surface $S$ enclosing $\Sigma$ is called a minimal area enclosure of $\Sigma$ with respect to $g$, if for all surfaces $T$ enclosing $\Sigma$

    \begin{equation*}
        |S|_g \leq |T|_g
    \end{equation*}
    where $| \cdot |_g$ is the mass measure with respect to the metric $g$.
\end{definition}

As is evident from the definition, multiple (possibly infinitely many) minimal area enclosures of $\Sigma$ can exist and all of them will have the same area. An \textit{outermost minimal area enclosure} of $\Sigma$ with respect to $g$, denoted $\ts_g$, is a minimal area enclosure of $\s$ with respect to $g$ that encloses every other minimal area enclosures of $\Sigma$ with respect to $g$. The following result about the existence of outermost minimal area enclosures is well known\cite{bray1999proof, huiskenilmanen}.

\begin{theorem}
    Let $(M,g)$ be an asymptotically flat 3-manifold with smooth, non-empty, compact boundary. Let $g$ extend smoothly to the boundary. Then, there exists a unique surface $\ts$ enclosing $\Sigma$ with the following properties:

    \begin{enumerate}
        \item $\ts$ is a minimal area enclosure of $\Sigma$ with respect to $g$.
        \item If S is any other minimal area enclosure of $\Sigma$ with repsect to $g$, then $\ts$ encloses $S$.
    \end{enumerate}
    Moreover, $\ts$ is an embedded $C^{1,1}$ surface with non-negative mean curvature, and $\ts \backslash \Sigma$, if non-empty, is $C^{\infty}$ with zero mean curvature.
\end{theorem}

It was proved in \cite{jeff}, that the set $\ts \cap \s$ has zero g-measure (c.f. Proposition 36). Since our proof of Theorem \ref{theorem: main_theorem} and its supporting lemmas relies on the local geometry of the manifold near points in the intersection of $\ts$ and $\s$, it is necessary to consider the following construction of a diffeomorphism between the neighbourhood of an arbitrary point $p \in \ts \cap \s$ and the upper half space. \cite{jeff} \\

Let $p \in \ts \cap \Sigma$, where $\ts$ denotes the outermost minimal area enclosure of $\Sigma$ with respect to the metric $g' = u^{4}g$. Choosing $\rho > 0$ sufficiently small, we construct a diffeomorphism $\Phi$ between the sets $\overline{U} = \overline{B(p,\rho)}$ and $\overline{B}^{+} = \{z \geq 0 \} \cap \overline{B(0,1)} \subset \mathbb{R}^3$ such that \\

\begin{enumerate}
    \item $\Phi(p) = (0,0,0)$
    \item $\Phi(\overline{B(p, \rho)} \cap \Sigma) = \{ z=0 \}$
    \item $d\Phi_{p} : T_{p}M \xrightarrow{} \mathbb{R}^{3}$ is an isometry.
\end{enumerate} 

The third property makes the coordinate frame $\{ \partial_{x}, \partial_{y}, \partial_{z} \}$ orthonormal at $p$. We identify the coordinates on the sets $\overline{B(p, \rho)}$ and $\overline{B}^{+}$ and denote them by $(x,y,z)$.  

For $\alpha > 0$, define

\begin{equation*}
    C_{\alpha} = \{ (x,y,z) \in \overline{U}: \sqrt{x^2 + y^2} \leq \alpha z \}.
\end{equation*}

$C_{\alpha}$ defines a region in $\overline{U}$ above a cone with opening parameter $\alpha$ and vertex at $p$. The cone closes as $\alpha \rightarrow 0$ and approaches a place as $\alpha \rightarrow +\infty$. We define \textit{lower partial density} as follows.

\begin{definition}
    Let $S \subset M$ be surface and let $p \in S \cap \Sigma$. We define the the lower partial density, denoted $\Theta_{\alpha}(S,p)$ of $S$ at point $p$ as:

    \begin{equation*}
        \Theta_{\alpha}(S,p) = \lim\limits_{r \rightarrow 0^{+}} \frac{|S \llcorner (B(p,r) \cap C_{\alpha})|_g}{\pi r^2}
    \end{equation*}
    
\end{definition}

Lower partial density is a measure of density of $S$ at $p$ with contributions coming only from parts within the cone $C_{\alpha}$. For contrast, the \textit{lower density} of $S$ at $p$ is defined as the quantity:

\begin{equation*}
    \Theta(S,p) = \lim\limits_{r \rightarrow 0^{+}} \frac{|S \llcorner B(p,r)|_g}{\pi r^2}
\end{equation*}

We can also define the \textit{partial mass function} as the numerator of the partial density function:

\begin{equation} \nonumber
    m_{\alpha,r}(S,p) = |S \llcorner (B(p,r) \cap C_{\alpha})|_g 
\end{equation}


\subsection{The generalised harmonic conformal class and maximizers}

Assume that $(M,g)$ is a smooth, complete, asymptotically flat 3-manifold with non-empty compact smooth boundary $\partial M = \Sigma$. Consider the following equivalence relation: 

\begin{equation*}
    g' \equiv g \Leftrightarrow \; g' = u^4g, \, \Delta_g u = 0, \, u > 0, \, u \rightarrow 1 \text{ at infinity }, \, u \text{ is smooth }
\end{equation*}

In particular, we identify metrics related to each other by well behaved harmonic functions $u$. The classes of equivalence under this identification are called the \textit{harmonic conformal classes} of $M$ \cite{bray1999proof}, and the harmonic conformal class to which $g$ belongs will be denoted by $\mathcal{H}(g)$. Every element of $\mathcal{H}(g)$ is an asymptotically flat metric, since the harmonic function $u$ expands as \cite{bartnik_mass} 

\begin{equation*}
    u(x) = 1 + \frac{a}{|x|} + O(\frac{1}{x^2})
\end{equation*}

Hence, all elements of $\mathcal{H}(g)$ have the same asymptotic behaviour. The asymptotic normalisation ($u \rightarrow 1$) in the above definition is for convinience. Having defined the conformal class, we can further generalise the above construction to define the \textit{generalised conformal class} \cite{jeff} of the metric $g$. We do this to generalise the set $\mathcal{H}(g)$ to some larger compact set $\overline{\mathcal{H}}(g) \supset \mathcal{H}(g)$ , in order to take maximums over it (used in the definition of maximizer below). \\ 

Let $f \geq 0$ belong to class $L^4(\Sigma)$ with respect to area measure $dA_g$ induces by $g$. For $x \in \text{int }M$ define 

\begin{equation*}
    u(x) = \varphi(x) + \int\limits_{\s} K(x,y) f(y) dA_g (y)
\end{equation*}

where $\varphi(x)$ is the unique g-harmonic function that vanishes on $\s$ and tends to one  at infinity, and $K(x,y)$ is the Poisson kernel for $(M,g)$ with $x \in M, \, y \in \s, \, x \neq y$. Hence, by definition, $u$ is g-harmonic in the interior of $M$ and tends to one at infinity. We will call $u$ as the harmonic function associated to $f$. Since $f$ is determined uniquely by $u$ (under almost everywhere equivalence), we will call $f$ to be the function in $L^4 (\s)$ determined by $u$. Since $u$ need not extend smoothly onto the boundary $\s$, $u^4g$ is a smooth metric on $M \backslash \s$.

\begin{definition}
    The \textbf{generalised harmonic conformal class} of $g$ is the set $\overline{\mathcal{H}(g)}$ of all Riemannian metrics $u^4g$ on $M \backslash \s$, where $u$ is the harmonic function associated to some non-negative $f \in L^4(\s)$.
\end{definition}

Having defined the generalised harmonic conformal class, we can define the area of surfaces $S$ enclosing the boundary $\s$ with respect to $g' = u^4g \in \overline{\mathcal{H}(g)}$ as 

\begin{equation*}
    |S|_{g'} = |S|_{u^4g} = \int\limits_{S \cap \s} f^4 d\hm^2_g + \int\limits_{S \backslash \s} u^4 d\hm^2_g
\end{equation*}

For our purposes, it is necessary to define the \textit{convergence of metrics} in $\gh$.

\begin{definition}
    Let $\{ u_n^4g \}$ be a sequence of metrics in $\gh$ associated with $\{ f_n \}$ of functions in $L^4(\s)$, and let $u^4g \in \gh$ be associated to $f \in L^4(\s)$. Then the sequence $\{ u_n^4g \}$ is said to \textbf{converge weakly} to the metric $u^4g$ if $\{ f_n \}$ converges weakly to $f$ in $L^4(\Sigma)$.
\end{definition}

Preliminary results regarding the topology of $\gh$ were proved in \cite{jeff}. We state here a result that will be helpful to prove the results in section \ref{section: positive_partial_density}.

\begin{lemma} (Lemma 15 of \cite{jeff}) 
    Suppose a sequence of metrics $\{ u_n^4g \}$ in $\gh$ converges weakly to $u^4g \in \gh$. Then $u_n$ converges pointwise to $u$ in the interior of $M$.
\end{lemma}

The $g$-area of a minimal area enclosure (if it exists) of $\Sigma$ with respect to $g$ is equal to the \textit{minimal enclosing area} of $\s$ with respect to $g$. More formally,

\begin{definition}
    For $g' \in \gh$, the \textbf{minimal enclosing area} of $\s$ with respect to $g'$, denoted min$(\s,g')$, is

    \begin{equation*}
        \text{min}(\s, g') = \inf\limits_{g'}\{ |S|_{g'} : \text{ $S$ encloses $\s$} \}.
    \end{equation*}
    
\end{definition}

\textit{Minimal area enclosure} and \textit{outermost minimal area enclosure} of $\s$ are defined as in the last section. At last, we can define the \textit{maximizer}.

\begin{definition}
    For $A>0$, we define 

    \begin{equation*}
        \alpha(A) = \sup\limits_{g' \in \gh} \{ \text{min}(\s,g'): |\s|_{g'} \leq A \}.
    \end{equation*}
    The metric $g'$ attaining the supremum in the above definition is called the \textbf{maximizer} for $\alpha(A)$
\end{definition}

The proof of the existence of the maximizer $g'$ above was presented in \cite{jeff} (Theorem 32).

\subsection{Convergence of surfaces and ZAS} \label{section: convergence and ZAS}

For other purposes and for defining zero area singularities (ZAS), we need to define the convergence of surfaces. Preliminary definition of convergence of surfaces, viewed as currents in $M$, was presented in section \ref{section: surfaces and mae}. In this section, we present a definition based on convergence of area measures.

\begin{definition}
    A sequence of surfaces $\{S_i = \s + \Omega_i\}$ is said to \textbf{converge} to a surface $S = \s + \Omega$ if the sequence $\{ \Omega_i - \Omega \}$ converges to zero in mass measure:

    \begin{equation*}
        |\Omega_i - \Omega|_g \rightarrow 0 
    \end{equation*}
    
\end{definition}

Stronger $C^k$ modes of convergence of $C^k$-surfaces were presented in \cite{Bray:2009voj}, mass convergence is implied by these $C^k$ convergence modes. This approach considers a convergence of surfaces to a surface $\s^{0} \subset M$ as a convergence of ``graph" of surfaces near $\s^{0}$ as follows: let $U \subset M$ be a neighbourhood of $\s^{0}$ diffeomorphic to $\s^{0} \times [0,a)$ for some $a > 0$. Let $(x,s)$ be coordinates on $U$ with $x \in \s^{0}$ and $s \in [0,a)$. We say that a surface $S \subset U$ is a \textit{graph} over $\s^{0}$ if it can be parametrised as $s = s(x)$ for $x \in \s^{0}$. Then the $C^k$ \textit{convergence} of surfaces in $U$ to $\s^{0}$ is defined as follows:

\begin{definition}
    A sequence of $C^{k}$ surfaces $\{ S_i \}$ that are graphs over $\s^{0}$ is said to \textbf{converge} to $\s^{0}$ in $C^k$ if the associated coordinate functions (see above) $s_n: \s^{0} \rightarrow [0,a)$ converges to zero in $C^k$.
\end{definition}

In \cite{Bray:2009voj}, authors have considered the surfaces to be $C^{\infty}$, although the definition stated above is equivalent. We can now define \textit{zero area singularities} as follows.

\begin{definition}
    Let $(M,g)$ be a 3-manifold with smooth, compact, nonempty boundary $\s$. Assume that $g$ is smooth on $M \backslash \s$. A connected component $\s^{0}$ of $\s$ is called a \textbf{zero area singularity (ZAS)} of $g$ if for every sequence of smooth surfaces $\{ S_i \}$ properly enclosing $\s^{0}$ and converging in $C^{1}$ to $\s^{0}$, the sequence of areas $\{ |S_n|_g \}$ converges to zero. 
\end{definition}

\section{Heuristic argument of the proof} \label{section: heuristic_proof}

Our proof of theorem \ref{theorem: main_theorem} can be conveniently split into two parts. It is a metric-variation argument over all surfaces $S$ enclosing the boundary component $\s$. Our proof is same as of Theorem 29 in \cite{jeff}, except that we ``kick out" the collection of surfaces enclosing $\s$ with \textit{zero partial density} from the main variational argument. We justify this removal of surfaces with zero lower partial density by proving that the outermost minimal area enclosure of $\s$ cannot have zero lower partial density. This is achieved in two steps. Firstly, let $p \in \ts \cap S$, then, in section \ref{lemma: zeropardensity_mgeq0} we prove that the lower partial density $\Theta_{\alpha}(S,p)$ of a surface $S$ is positive under the assumption that its numerator (i.e., the \textit{partial mass function}) is not trivially zero. The argument for proving this is based on Lemma 25 of \cite{jeff}, with introduction of new techniques for defining \textit{partial slices of surfaces}, and their properties thereby. Secondly, in section \ref{lemma: zeropardensity_meq0} we prove that the partial mass function $m_{\alpha,r}(\ts,p)$ (defined as the numerator of $\Theta_{\alpha}(\ts,p)$) of an outermost minimal area enclosure $\ts$ of $\s$ cannot be trivially zero by using a coarea argument. These two arguments imply that that lower partial density of an outermost minimal area enclosure $\ts$ of $\s$ is positive, and hence we ``remove'' all surfaces with zero partial density from the variational argument in our proof of the main theorem. A more rigorous statement is provided in the proof of the main theorem.\\ 

This leaves us with a variational argument consisting of enclosures of $\s$ with positive lower partial density: $\Theta_{\alpha}(S,p) > 0$. The proof of this final variational argument depends upon proving the positivity of a quantity containing integrals uniformly over all surfaces $S$ with positive lower partial density $\Theta_{\alpha}(S,p)$. This requires the use of a uniform version of Fatou's lemma. We prove this modified version of Fatou's lemma in section \ref{section: positive_partial_density} using results obtained by the authors in \cite{feinberg2015uniform}. By the weak convergence of mass measures as proved in \cite{jeff}, and the modified version of the uniform Fatou's lemma, the positivity of the aforesaid quantity follows. This completes our proof of the main theorem.\\ 

At last, let us discuss the metric-variation argument of the main theorem more explicitly. The idea was presented originally in Theorem 29 of \cite{jeff}. We start with a specially chosen path $g_t = u_t^4g$ of metrics in $\gh$, such that the area of all minimal area enclosures of $\s$ with positive lower partial density increases at a uniformly positive rate, and the area of $\s$ is unchanged up to first order in variation along the path. We can further normalise this path to a path $\overline{g_t}$ in $\gh$, for which the minimal enclosing area increases strictly for sufficiently small $t>0$. This contradicts the assumption that $\overline{g_o} = g'$ is a maximizer for $\alpha(S)$.

\section{Proof of the main theorem} \label{section: proof}

\subsection{Part 1: Positive lower partial density} \label{section: positive_partial_density}

In this section we prove the following generalisation of Proposition 40 from \cite{jeff} for the case of minimal area enclosures with positive lower partial density, which will be used to prove theorem \ref{theorem: main_theorem} in section \ref{section: proof}.

\begin{proposition} \label{prop:Prop_1}
    Suppose $\alpha ( A )  < A$, and let $g' = u^{4}g$ be a maximizer for $\alpha ( A )$. If $\{S_i\}$ is family of a minimal area enclosures of $\Sigma$ with respect to $g'$ such that the lower partial density $\Theta_{\alpha}(S_i,p)$ is positive for all $r \in (0,\rho_i]$ (where $\rho_i > 0$ is chosen to be sufficiently small) and for some $\alpha_i > 0$ for each $i$, then: 
    \begin{equation} \nonumber
        \liminf\limits_{\sigma \xrightarrow{} 0^{+}} \left( \inf\limits_{S} \left(  \int\limits_{S\backslash\Sigma} u^{3}w_{\sigma}d\hm^{2}_{g} \right) \right) \geq 1
    \end{equation}
    Moreover, there exists $\sigma > 0$ such that the following holds:
    \begin{equation} \label{eqn: prop1_ineq}
         \int\limits_{S\backslash\Sigma} u^{3}w_{\sigma}d\hm^{2}_{g} - \sqrt{\frac{\alpha(A)}{A}} > 0
    \end{equation}
    uniformly for all minimal area enclosures $S \in \{S_i\}$.
\end{proposition}

The generalisation lies in the fact that we prove the above proposition uniformly over all surfaces $S_i$, allowing us to keep the second quantity positive even if the count of distinct surfaces $S_i$ is infinite. Our motivation comes from\cite{song}, in which the author proved the existence of infinitely many minimal embedded surfaces in closed 3-manifolds. \\

For our proof of Proposition \ref{prop:Prop_1}, it is necessary to define a function $h_{\sigma}$ on $\Sigma$ as follows. \cite{jeff}

\begin{equation}
    h_{\sigma} = \frac{1}{\pi \sigma^{2}}\chi_{B(p,\sigma)},
\end{equation}

where $\chi$ is the characteristic function. As is evident,

\begin{equation}
    \lim\limits_{\sigma \xrightarrow{} 0^{+}}\int_{\Sigma}h_{\sigma}d\hm^{2} = 1
\end{equation}

Therefore for sufficiently small $\sigma > 0$, we have 

\begin{equation} \label{eqn3}
   \int_{\Sigma}h_{\sigma}d\hm^{2} \leq \sqrt{\frac{A}{\alpha(A)}}.
\end{equation}

Define $k_{\sigma} = \frac{\int_{\Sigma}h_{\sigma}d\hm^{2}}{A}$ and for $t \geq 0$. Let

\begin{equation}
    f_t = f + t\left( \frac{h_{\sigma}}{f^3} - k_{\sigma}f \right).
\end{equation}

Since $f \geq \epsilon$, for some $\epsilon > 0$, by Proposition 33 of \cite{jeff}, we obtain the boundedness of $1 /f^3$.

From \eqref{eqn3}, we have for sufficiently small $\sigma$ 

\begin{equation}
\label{eqn5}
    k_{\sigma} \leq \frac{1}{A} \sqrt{\frac{A}{\alpha (A)}}
\end{equation}

We define $u_t$ to be the harmonic function associated to $f_t$:

\begin{equation}
    u_t = u + t (w_{\sigma} - k_{\sigma}(u-\varphi))
\end{equation}

where $w_{\sigma}$ is the harmonic function tending to 0 at infinity and given by $h_{\sigma}/f^3$ on $\Sigma$. By the choice of $k_{\sigma}$, which depends on the minimal enclosure $S$, the area of $\Sigma$ is unchanged to first order at 0.\\

For any minimal area enclosure $S_i$ of $\Sigma$, using \eqref{eqn5}, $|S|_{g'} = \alpha(A)$ and $k_{\sigma}u^3\varphi > 0$, we have that for sufficiently small $\sigma$ - 

\begin{gather*} 
\label{eqn7}
    \frac{1}{4}\frac{d}{dt}|_{t=0} |S_i|_{g_t} = \int_{S_i \cap \Sigma} f^3 \left( \frac{h_{\sigma}}{f^3} - k_{\sigma}f \right)d\hm^{2}_{g} + \int_{S_i \backslash \Sigma} u^3 \left( w_{\sigma} - k_{\sigma}u + k_{\sigma}\varphi \right)d\hm^2_g \\ 
    = \int\limits_{S_i \cap \Sigma} h_{\sigma}d\hm^2_g + \int\limits_{S_i \backslash \Sigma} u^3 w_{\sigma}d\hm^2_g \\ 
    - k_{\sigma} \left( \int\limits_{S_i \cap \Sigma} f^4 d\hm^2_g + \int\limits_{S_i \backslash \Sigma} u^4 d\hm^2_g \right) + k_{\sigma} \int\limits_{S_i \backslash \Sigma} u^3 \varphi d\hm^2_g  \\
    > \int_{S_i \cap \Sigma} h_{\sigma}d\hm^2_g + \int_{S_i \backslash \Sigma} u^3w_{\sigma}d\hm^2_g - \sqrt{\frac{\alpha (A) }{A}} \numberthis 
\end{gather*}

We prove by Proposition \ref{prop:Prop_1} that the $\sigma$ in the last equality can be chosen small enough so that \eqref{eqn: prop1_ineq} is positive for all surfaces uniformly. \\

\begin{definition}
    Total variation (TV) of measures $\mu$ and $\nu$, uniformly over a collection of minimal area enclosures $\{ S_i \}$ as  - 
    \begin{equation} \nonumber
        \nonumber
        ||\mu - \nu||_{TV} = \sup\limits_{S \in \{S_i\}} \left( \int_{S} f(s)  d\mu - \int_{S} f(s) d\nu | f(s): \mathbb{S} \rightarrow [-1,1] \text{ is measurable } \right)
    \end{equation}
\end{definition}

To obtain the aforesaid uniformly positive estimate in \eqref{eqn: prop1_ineq}, we need to use a uniform version of the classical Fatou's lemma from \cite{feinberg2015uniform}.

\begin{theorem} \label{theorem: feinberg}
    Let $(\mathbb{S}, \Sigma)$ be a measurable space. Let $M(\mathbb{S})$ be set of finite measures on $\mathbb{S}$. Let $\{ \mu^{(n)}\}_{n=1,2,...} \subset M(S)$ converge in total variation to a measure $\mu$ on $\mathbb{S}$, $f \in L^1(\mathbb{S},\mu)$, and $f_n \in L^1(\mathbb{S},\mu^{(n)})$ for each $n=1,2,...$. Then the inequality 

    \begin{equation} \nonumber 
        \liminf\limits_{n \rightarrow \infty}\inf\limits_{S \in \Sigma} \bigg( \int_S f_n(s) \mu^{(n)}(ds) - \int_S f(s) \mu(ds) \bigg) \geq 0
    \end{equation}

    holds if and only if the following two statements hold:

    \begin{enumerate}
        \item for each $\varepsilon > 0$
        \begin{equation} \nonumber 
            \mu(\{ s \in \mathbb{S}: f^n(s) < f(s) - \varepsilon \}) \rightarrow 0 \text{ as } n \rightarrow \infty
        \end{equation}
        \item
        \begin{equation} \nonumber 
            \liminf\limits_{K \rightarrow +\infty} \inf\limits_{n=1,2,...} \int_s f^n(s) \boldsymbol{I}\{ s \in \mathbb{S}: f^n(s) \leq -K \} \mu^n(ds) \geq 0
        \end{equation} 
    \end{enumerate}
    
\end{theorem}

We remark that the above theorem is true with the assumption that $\lim\limits_{n \rightarrow \infty} f_n = f$ uniformly, and that $f_n$ and $f$ are non-negative functions for each $n=1,2,3,...$ while relaxing the condition on the integrability of functions $f$ and $f_n$. This results in a more simpler uniformity in the classical Fatou's lemma, which is sufficient for our purposes. 

\begin{theorem} \label{theorem: modified_fatou} (Modified Uniform Fatou's Lemma) 
    Let $(\mathbb{S}, \Sigma)$ be a measurable space. Let $\mu$ be a finite measure on $\mathbb{S}$. Let $\lim\limits_{n \rightarrow \infty}f_n(x) = f(x)$ uniformly. Then,

    \begin{equation} \label{eqn: uniform_fatou}
        \liminf\limits_{n \rightarrow \infty}\inf\limits_{S \in \Sigma} \bigg( \int_S f_n(s) \mu(ds) - \int_S f(s) \mu(ds) \bigg) \geq 0 \\
    \end{equation}
\end{theorem}

\begin{proof}

    Since $\lim\limits_{n \rightarrow \infty}f_n = f$,
    
    \begin{equation*}
        \forall \; \varepsilon > 0 \exists N \in \mathbb{N}: \forall n \geq N \; \sup\limits_{x \in \mathbb{S}}|f_n(x) - f(x)| < \varepsilon.
    \end{equation*}

    The proof them immediately follows from Lemma 2.5 of \cite{feinberg2015uniform}. Here, we present a condensed version of the argument. \\ 

    Following the construction at the beginning of the aforesaid lemma in \cite{feinberg2015uniform}, it can be shown that inequality \ref{eqn: uniform_fatou} is equivalent to the following condition:
    
    \begin{equation*}
        \liminf\limits_{K \rightarrow +\infty} \liminf\limits_{n \rightarrow \infty} I(n,K) \geq 0
    \end{equation*}

    where $I(n,K) := \inf\limits_{S \in \Sigma} \bigg( \int\limits_{S_{f^{(n)} > -K}} \Big( f^{(n)}(s) - f(s) \Big) d\mu(s) \bigg) $

    Following \cite{feinberg2015uniform}, the above integral $I$ can be split into three parts:

    \begin{equation*}
        I(n,K) = I_1 (n,K) + I_2 (n,K) + I_3 (n,K)
    \end{equation*}

    where, in our case ($\mu^{n} \equiv \mu$) 

    \begin{gather*}
        I_1 (n,K) = 0 \\ 
        I_2 (n,K) := \inf\limits_{S \in \Sigma} \bigg( \int\limits_{S_{|f^{(n)|} < K}} \Big( f^{(n)}(s) - f(s) \Big) d\mu(s) \bigg) \\ 
        I_3 (n,K) := \inf\limits_{S \in \Sigma} \bigg( \int\limits_{S_{f^{(n)}\geq K}} \Big( f^{(n)}(s) - f(s) \Big) d\mu(s) \bigg)
    \end{gather*}

    We estimate the lower bound for $I_3$ as follows:

    \begin{equation*}
        I_3 (n,K) \geq 0
    \end{equation*}

    where we have used the analogous inequality from \cite{feinberg2015uniform} and $\mu^n \equiv \mu$. Hence,

    \begin{equation*}
        \liminf\limits_{K \rightarrow +\infty}\liminf\limits_{n \rightarrow \infty} I_3 (n,K) \geq 0
    \end{equation*}

    Now, for $I_2 (n,K)$, we have

    \begin{gather*}
        I_2 (n,K) \geq - \varepsilon \mu(\mathbb{S}) - K \mu(S_{f - f^{(n)} > \varepsilon}) - \int\limits_{\mathbb{S}_{f - f^{(n)} > \varepsilon}} |f(s)| \mu(ds) 
    \end{gather*}

    which under the limits over $n$ and $K$, considering the \textit{uniform convergence} of $\{f_n\}$ to $f$, gives

    \begin{equation*}
        \liminf\limits_{K \rightarrow +\infty}\liminf\limits_{n \rightarrow \infty} I_2(n,K) \geq 0
    \end{equation*}

    Hence, we get 

    \begin{gather*}
        \liminf\limits_{K \rightarrow +\infty}\liminf\limits_{n \rightarrow \infty} I(n,k) = \liminf\limits_{K \rightarrow +\infty}\liminf\limits_{n \rightarrow \infty} I_1(n,k) + \\ \liminf\limits_{K \rightarrow +\infty}\liminf\limits_{n \rightarrow \infty} I_2(n,k) + \liminf\limits_{K \rightarrow +\infty}\liminf\limits_{n \rightarrow \infty} I_3(n,k) \geq 0
    \end{gather*}
    
\end{proof}

We apply the above result to the set of minimal area enclosures of $\s$ with positive partial density. 

\begin{lemma} \label{lemma: uniform_ineq}
    Let each $S_i$ be a minimal area enclosure of $\Sigma$ with positive partial density. Let $X \subset\subset int(M)$. Then,

    \begin{equation}
        \liminf\limits_{\sigma \xrightarrow{} 0^{+}} \left( \inf\limits_{S \in \{S_i\}} \left( 
        \int\limits_{(S \backslash \Sigma) \cap X} v_{\sigma}d\hm^2_g - \int\limits_{(S \backslash \Sigma) \cap X} K (.;p)d\hm^2_g (.) \right) \right) \geq 0
    \end{equation}

    where $v_{\sigma}$ is the harmonic function given by $h_{\sigma}$ on $\Sigma$ and 0 at infinity.
    
\end{lemma}

To prove the above result, it is necessary to establish the uniform convergence of the functions $u_{\sigma}$ appearing in the above lemma to the Poisson kernel based at $p$ - $K(x,p)$ over sets compactly contained in $int(M)$. We prove this result in three parts. Firstly, we prove a lemma similar to Lemma 15 of \cite{jeff}, giving the condition for uniform convergence of the harmonic functions $\{u_{\sigma}\} \subset \overline{\mathcal{H}(\Sigma)}$ over sets compactly contained in the interior of $M$. Secondly, we prove the \textit{strong convergence} of the series of measures $h_{\sigma} d\hm^2_g$ to the $\delta_p$-measure. Finally, we prove that the series of functions $u_{\sigma}$ converges to the Poisson kernel based at $p$ over all sets compactly contained in the interior of $M$.

\begin{definition}
    Let $\{ u_n g \}$ be a series of measures in $\overline{\mathcal{H}(\Sigma)}$, and let $u g \in \overline{\mathcal{H}(\Sigma)}$. We say that $\{ u_n g \}$ \textit{converges strongly in $C^{k}$} to $ug$ if the sequence of associated functions $\{ f_n \}$ converges to $f$ (associated to u) distributionally in $C^k$:

    \begin{equation} \nonumber 
        \sup\limits_{g \in C^k(\Sigma)} \bigg| \int\limits_{\Sigma} g(x) (f_{n}(x) - f(x)) d\hm^2_g \bigg| \xrightarrow{n \rightarrow \infty} 0.
    \end{equation}
    
\end{definition}

\begin{lemma} \label{lemma: strong convergence of u}
    If $\{ u_n g \} \subset \overline{\mathcal{H}(\Sigma)}$ converge strongly in $C^1$ to $ug \in \overline{\mathcal{H}(\Sigma)}$, then $u_n \rightarrow u$ uniformly over sets compactly contained in $int (M)$.
\end{lemma}

\begin{proof}

    \begin{equation*}
        \sup\limits_{X \subset \subset int(M)} \Big| u_n(x) - u(x) \Big| = \sup\limits_{X \subset\subset int(M)} \bigg| \int\limits_{\Sigma} K(x,y) (f_{n}(y) - f(y)) dA_{y} \bigg|
    \end{equation*}

    Since for each $x \in int(M)$, $y \mapsto K(x,y) \in C^1(\Sigma)$ (since the kernel itself belongs to $C^2$), convergence of the above supremum follows because

    \begin{equation*}
        \sup\limits_{h \in C^1(\Sigma)} \bigg| \int\limits_{\Sigma} h(x) (f_n(x) - f(x)) dA_y \bigg| \xrightarrow{n \rightarrow \infty} 0
    \end{equation*}

    over any compactly contained set inside the interior of $M$.
    
\end{proof}

\begin{lemma} \label{lemma: distributional conv h_sigma}
    The sequence of measures $\{ h_{\sigma} \}$ converges strongly in distributional sense in $C^1$ to the $\delta$-measure based at $p$, as $\sigma \rightarrow 0^{+}$, over any set compactly contained inside $int(M)$.
\end{lemma}

\begin{proof}
    Let $X \subset\subset M$.
    \begin{gather*} 
        \sup\limits_{g \in C^1(\Sigma)} \bigg| \int\limits_{\Sigma \cap X} g(y) (h_{\sigma}(y) - \delta_{p}(y)) dA_y \bigg| \\ 
        = \sup\limits_{g \in C^1(\Sigma)} \bigg| \int\limits_{\Sigma \cap X} g(y) \frac{\chi_{B(p,\sigma) \cap \Sigma}}{\pi \sigma^2} dA_y - \int\limits_{\Sigma \cap X} g(y) \delta_p(y) dA_y \bigg| \\ 
        = \sup\limits_{g \in C^1(\Sigma)} \bigg| \int\limits_{B(p,\sigma) \cap \Sigma \cap X} \frac{g(y)}{\pi \sigma^2} dA_y - g(p) \bigg| \\ 
        = \sup\limits_{g \in C^1(\Sigma)} \bigg| \frac{1}{\pi \sigma^2} \int\limits_{B(p,\sigma) \cap \Sigma \cap X} (g(y) - g(p)) dA_y \bigg| \\ 
        \leq \sup\limits_{g \in C^1(\Sigma)} \bigg( \sup\limits_{y \in B(p,\sigma) \cap X} \Big| g(y) - g(p) \Big| \bigg)
        \leq \sup\limits_{g \in C^1(\Sigma)} \bigg( \sup\limits_{y \in \overline{B(p,\sigma) \cap X}} |g'(y)| 2\sigma \bigg)
    \end{gather*}

    which tends to zero uniformly as we let $\sigma \rightarrow 0^{+}$. Here, we have used the fact that $g'$ is continuous and $\sigma \in [0,\rho)$.
    
\end{proof}

\begin{corr} \label{corollary: corr1}
    $v_{\sigma} \xrightarrow{\sigma \rightarrow 0^{+}} K(\cdot ,p)$ uniformly over sets compactly contained in $int(M)$.
\end{corr}

\begin{proof}
    Follows immediately from Lemma \ref{lemma: strong convergence of u} and \ref{lemma: distributional conv h_sigma}.
\end{proof}

We can now prove Lemma 2.

\begin{proof} (Proof of Lemma \ref{lemma: uniform_ineq})
    By Corollary \ref{corollary: corr1}, $v_{\sigma}$ converges uniformly over sets compactly contained in $int(M)$ to $K(x,p)$ - the Poisson kernel based at $p$. The proof then follows from Theorem \ref{theorem: modified_fatou}.
\end{proof}

We can now prove Proposition \ref{prop:Prop_1}.

\begin{proof}(Proof of Proposition \ref{prop:Prop_1})
    Since $v_{\sigma}$ is dominated by $u^3w_{\sigma}$ over M, it suffices to show 
    
    \begin{equation} \nonumber
        \liminf\limits_{\sigma \xrightarrow{} 0^{+}} \left( \inf\limits_{S_i} \int\limits_{S_i \backslash \Sigma} v_{\sigma}d\hm^2_g \right) = +\infty.
    \end{equation}
    
    From Theorem \ref{lemma: uniform_ineq}, for any $X \subset\subset int(M)$, we have

    \begin{equation}\label{eq:Prop_1}
        \liminf\limits_{\sigma \xrightarrow{} 0^{+}} \left( \inf\limits_{S \in \{S_i\}} \left( 
        \int\limits_{(S \backslash \Sigma) \cap X} v_{\sigma}d\hm^2_g - \int\limits_{(S \backslash \Sigma) \cap X} K (.;p)d\hm^2_g (.) \right) \right) \geq 0.
    \end{equation}
    
    We wish to use the construction of Lemma 37 in \cite{jeff}, showing that the integral of Poisson kernel over a cone in the upper half space is equal to infinity. We will do this by restricting our analysis to regions compactly contained in interior of $M$ and then taking limit to approach $\partial M$ from inside. To that end, let $\{X_i\} \subset\subset int(M)$ be a sequence of regions in $M$ such that $X_i \subset X_j$ whenever $j > i$ and distance between $\partial M$ and $\partial X_i$ is less than some small enough $\epsilon$ for each $i$. We also require $dist(\partial M, \partial X_i) \rightarrow 0$ as we let $i \rightarrow \infty$. First we estimate the integral of Poisson kernel over sets $(S \backslash \s) \cap X_k$ as follows.

    \begin{align*}
        \int\limits_{((S \backslash \Sigma) \cap X_k) \cap B(p,r)} K (.;p)d\hm^2_g (.) \geq \int\limits_{(S \cap X_k) \cap B(p,r) \cap C_{\alpha}} \frac{z}{(x^2 + y^2 + z^2)^{3/2}} d\hm^2_g \\
        \geq \int\limits_{(S \cap X_k) \cap B(p,r) \cap C_{\alpha}} \frac{z}{(\alpha^2 z^2 + z^2)^{3/2}} d\hm^2_g \\
        \geq \int\limits_{(S \cap X_k) \cap B(p,r) \cap C_{\alpha}} \frac{z}{(r^2 z^2 + z^2)^{3/2}} d\hm^2_g \\
        \geq \int\limits_{(S \cap X_k) \cap B(p,r) \cap C_{\alpha}} \frac{1}{(\alpha^2 z^2 + z^2)^{3/2}} d\hm^2_g \\
        = \frac{\pi}{(\alpha +1)^{3/2}} \frac{\hm^2_g((S \cap X_k) \cap B(p,r) \cap C_{\alpha})}{\pi r^2} \\
        = \frac{\pi}{(\alpha +1)^{3/2}} \bigg( \frac{\hm^2_g((S \cap C_{\alpha}) \cap B(p,r))}{\pi r^2} - \frac{\hm^2_g((S \cap C_{\alpha}) \cap B(p,r) \backslash X_k)}{\pi r^2} \bigg).
    \end{align*}

    By the definition of lower partial density,

    \begin{equation*}
        \frac{\hm^2_g((S \cap C_{\alpha}) \cap B(p,r))}{\pi r^2} \geq \Theta_{\alpha}(S, p) / 2.
    \end{equation*}

    Because of our choice of the sets $X_k$ (we can take $\epsilon$ arbitrarily small, in particular, smaller than $\frac{\Theta_{\alpha}(S,p)}{2}$), the last quantity can be made positive. Taking monotonous limit over increasing sets $\{X_k\}$ we get lower bound for $\lim\limits_{k \rightarrow \infty} \int\limits_{((S \backslash \Sigma) \cap X_k \cap B(p,r))} K (.;p)d\hm^2_g (.) = \int\limits_{((S \backslash \Sigma) \cap B(p,r))} K (.;p)d\hm^2_g (.)$ independent of $r$, hence we have 

    \begin{equation}
        \int\limits_{S \backslash \Sigma} K (,;p)d\hm^2_g(.) = + \infty.
    \end{equation}

    Also, note that 

    \begin{equation*}
        \lim\limits_{k \rightarrow \infty} \int \limits_{(S \backslash \s) \cap X_k} v_{\sigma} d\hm^2_g = \int\limits_{S \backslash \s} v_{\sigma} d\hm^2_g
    \end{equation*}

    since the integrand is independent of sets $X_k$, which form a monotonous sequence with boundary converging to $\partial M$. We pass to the limit $k \rightarrow \infty$ in inequality \eqref{eq:Prop_1} (which is valid for each $X_k$) to complete the proof.
    
\end{proof}

\subsection{Part 2: Zero lower partial density}

We remind that $\overline{U} = \overline{B}(p,\rho)$ and $\overline{B}^{+} = \overline{B}(0,1) \cap \{ z \geq 0 \}$. Also, we identify the coordinates on $\overline{U}$ and $\overline{B}^{+}$ via the diffeomorphism between them ($\Phi$). Recall that embedded, oriented $C^1$ surfaces $S \subset M$ define a 2-current as 

    \begin{equation} \nonumber 
        S(\omega) = \int\limits_{S}\omega
    \end{equation}

    where $\omega$ is a compactly supported smooth differential 2-form. To agree with Stokes' theorem, the boundary current $\partial S$ is defined as 

    \begin{equation} \nonumber 
        \partial S(\phi) = \int\limits_{S}d\phi
    \end{equation}
     where $\phi$ is a 1-form.

From hereon, by a surface we mean an oriented, embedded $C^1$ surface in $M$. The definition of the restriction of current given by the surface $S$ to another open set $E$ with indicator function $\varphi$, can be dualised in the following sense \cite{simon}.

    \begin{equation}
        (S \llcorner E)(\omega) = \int\limits_{S \cap E}\omega = \int\limits_{ E}\omega \;(\text{if } E \subset S)
    \end{equation}
    \begin{equation}
        (S \llcorner \varphi)(\omega) = \int\limits_{S}\varphi(x)\omega
    \end{equation}
    
We will call these \textit{set representation (set rep for shorthand)} and \textit{functional representation (func rep for shorthand)} of the restriction of the current $S$ respectively. \\

Next, let $f$ be a Lipschitz function defined on some open set U whose values lie in $\mathbb{R}^{+}$. We construct a function that \textit{``generates"} the set of type $\{f < a\}, \: a \in \mathbb{R}^{+}$. Let $\epsilon$ be arbitrary, and let $\gamma$ be a smooth increasing function on $\mathbb{R}$ whose values lie in $\mathbb{R}^{+}$ such that - 

\begin{equation}
    \gamma(t) = 1, \; t < b, \; \gamma(t) = 0, \; t > a, \; 0 \geq \gamma'(t) \geq \frac{-(1 + \epsilon)}{b - a}, \;  b < t < a 
\end{equation}

Then by the above construction, the function $\gamma \circ f$ tends smoothly to the indicator of the set $\{ f < a\}$ as we let $b \uparrow a$. We say that $\gamma \circ f$ \textit{generates the set} $\{f < a\}$. \\

\begin{definition}
    Let $U \subset M$ be a surface, f be a Lipschitz function and $C_{\alpha}$ be the cone with opening parameter $\alpha$, as defined above. Then, we define the \textit{partial slice of U at level $r \in \mathbb{R}$ and conic parameter $\alpha$} as:

    \begin{equation} \nonumber
        \langle U,f,r,\alpha \rangle = \partial(U \llcorner (\{ f < r \} \cap C_{\alpha})) - (\partial U) \llcorner (\{f < r \} \cap C_{\alpha})
    \end{equation}
\end{definition}

The definition of the partial slice resembles the definition of the \textit{slice of a current} \cite{simon}, except the intersection with the cone $C_{\alpha}$. Analogous to the construction of $\gamma$ to generate the set $\{ f < a \}$, let us define $\varphi_z (\zeta)$ as 

    \begin{eqnarray} \nonumber
        \varphi_z (\zeta) = \begin{cases}
            0 & \zeta > \alpha z \\
            1 & \zeta < b'
        \end{cases} \\ \nonumber
        0 \geq \varphi_z ' (\zeta) \geq -\frac{1+\epsilon}{b' - \alpha z}, \;\; b' < \zeta < \alpha z 
    \end{eqnarray}
    
    for fixed $\alpha, z, \text{ and } b' < \alpha z$. \\
    
    Let $\tilde{f} = \sqrt{x^2 + y^2}$ be a 1-Lipschitz function. Note that
    
    \begin{equation} \nonumber
        C_{\alpha} = \{ (x,y,z) \in \overline{U} | \sqrt{x^2 + y^2} \leq \alpha z \} = \bigcup\limits_{1 \geq z \geq 0} \{ \tilde{f} < \alpha z \}.
    \end{equation}
    
That is, the cone with parameter $\alpha$ is generated by stacking its level sets.The necessity of such a stacking will be evident in the subsequent lemma. The answer boils down to the fact that by using such a stacking, at each level $z = z_0$, we get a 1-Lipschitz function that generates the corresponding slice of the disk, namely the function $\tilde{f} = \sqrt{x^2 + y^2}$. We note that $\varphi_z \circ \tilde{f} : \overline{U} \xrightarrow{} \mathbb{R}^{+} $ generates the set $C_{\alpha}$ as we take the limit $b' \uparrow \alpha z$ and integrate the composition $\varphi_z \circ \tilde{f}$ from 0 to $\rho$ over the parameter $z$. To see this, note the following:

    \begin{equation} \nonumber 
        \int\limits_{U \llcorner C_{\alpha}} \omega  = \int\limits_{0}^{1} dz \left( \int\limits_{U}\lim\limits_{b' \uparrow \alpha z}(\varphi_z \circ \tilde{f}) \omega \right) \stackrel{Fubini}{=\joinrel=\joinrel=} \int\limits_{U}\left( \int\limits_{0}^{1} \lim\limits_{b' \uparrow \alpha z}(\varphi_z \circ \tilde{f}) dz \right)\omega
    \end{equation}
    
    for any 1-form $\omega: |\omega| \leq 1$. 
    
\begin{lemma} \label{lemma: slice algebra}
    Let $U$ be a surface in $M$, f be a Lipschitz function on $\overline{B}^{+}$, $a \in \mathbb{R}^{+}$, $\tilde{f} = \sqrt{x^2 + y^2}$, and $\varphi_z$ be defined as above for some fixed $\alpha > 0$. Let us denote $C_{\alpha} = \int\limits_{0}^{1} (\varphi_z \circ \tilde{f}) dz$. Let $\omega$ be a 1-form defined on the set U. Let $\gamma$ be a smooth function defined on $\mathbb{R}^{+}$ such that its composition $\gamma \circ f$ generates the set $\{f < a\}$ (as above). Then as $b \uparrow a$,
    \begin{align*}
        \partial (U \llcorner (\{f < a\} \cap C_{\alpha})) (\omega) - (\partial U) \llcorner (\{f < a\} \cap C_{\alpha})(\omega) = \\ 
        = - (U \llcorner C_{\alpha})(\gamma'(f)df \wedge \omega) + \int\limits_{0}^{1} dz \Big( (U \llcorner \{ f < a \})(\omega \wedge (\varphi_{z}'(\tilde{f})d\tilde{f})) \Big)
    \end{align*}
    
\begin{remark} \nonumber
    Here the symbol $C_{\alpha}$ is used to denote the function composition $\int\limits_{0}^{1} (\varphi_z \circ \tilde{f}) dz$ without the limit $b' \uparrow \alpha z$ that needs to be taken under the integral sign. We therefore, prove this lemma for the ``broken" limit of the cone at point $p \in \ts \cap \Sigma$. We will let $b' \uparrow \alpha z$ under the integral sign in the subsequent lemma. This has been done for convenience.
\end{remark}
    
\end{lemma}

\begin{proof}

    \begin{align*}
        \partial(U \llcorner (\{f < a\} \cap C_{\alpha}))(\omega) - ((\partial U) \llcorner (\{f < a\} \cap C_{\alpha}))(\omega) = \\
        = \partial((U \llcorner C_{\alpha}) \llcorner \{f < a \})(\omega) - ((\partial U \llcorner C_{\alpha}) \llcorner \{f < a \})(\omega) \; \text{(Use set rep)} \\
        = \lim\limits_{b \uparrow a} \Big[ \left( U \llcorner C_{\alpha})(( \gamma \circ f )(d\omega)) - (\partial U \llcorner C_{\alpha})(( \gamma \circ f )(\omega)\right) \Big] \label{eqn: bound_cur} \numberthis \\
    \end{align*}
    
    where in the last line we have used the fact that $\gamma \circ f$ generates $\{f < a\}$, and the definition of boundary current. We now evaluate the second term in the expression above.
    
    \begin{gather*}
        (\partial U \llcorner C_{\alpha})((\gamma \circ f)\omega) = \int\limits_{\partial U \cap C_{\alpha}}(\gamma \circ f)\omega = \int\limits_{\partial U}\left( (\gamma \circ f)\omega \int\limits_{0}^{1}(\varphi_z \circ \tilde{f})dz  \right) \\ \stackrel{Fubini}{=\joinrel=\joinrel=} \int\limits_{0}^{1} dz \left( \int\limits_{\partial U}(\gamma \circ f)\omega (\varphi_z \circ \tilde{f})  \right)
    \end{gather*}
    
    The above chain of equations can be understood as the boundary of the surface $U$ intersecting the cone $C_{\alpha}$, for some fixed $\alpha$, being generated by unifying the slices of boundary of surface $U$, intersecting each level set of the cone at fixed $z$, over all values of $z \in [0,1]$. Continuing this chain further, we have
    
    \begin{gather*}
        \stackrel{Stokes'}{=\joinrel=\joinrel=} \int\limits_{0}^{1} dz\left( \int\limits_{U} d[(\varphi_z \circ \tilde{f})((\gamma \circ f)\omega)]) \right) \\ 
        = \int\limits_{0}^{1}dz \left[ \int\limits_{U}d((\gamma \circ f)\omega) (\varphi_z \circ \tilde{f}) + (-1)^1 \int\limits_{U}((\gamma \circ f)\omega) \wedge d(\varphi_z \circ \tilde{f}) \right] \\ 
        = \int\limits_{0}^{1} dz\int\limits_{U}d((\gamma \circ f)\omega) (\varphi_z \circ \tilde{f}) - \int\limits_{0}^{1} dz\int\limits_{U} ((\gamma \circ f)\omega) \wedge \varphi_{z}'(\tilde{f})d\tilde{f}
    \end{gather*}
    
    Putting this back into \ref{eqn: bound_cur}, we get
    
    \begin{gather*}
        \lim\limits_{b \uparrow a} \bigg( ((U \llcorner C_{\alpha}))((\gamma \circ f)d\omega) - \int\limits_{0}^{1} dz \int\limits_{U}d((\gamma \circ f)\omega)(\varphi_z \circ \tilde{f}) \\ + \int\limits_{0}^{1} dz \int\limits_{U} ((\gamma \circ f)\omega) \wedge \varphi_{z}'(\tilde{f})d\tilde{f}  
        \bigg) \\
        = \lim\limits_{b \uparrow a} \bigg( ((U \llcorner C_{\alpha}))((\gamma \circ f)d\omega) - \int\limits_{U} d((\gamma \circ f)\omega) \Big( \int\limits_{0}^{1} dz (\varphi_z \circ \tilde{f}) \Big) \\  + \int\limits_{0}^{1} dz \Big( (U \llcorner \{ f < a \}) \omega \wedge (\varphi_{z}'(\tilde{f})d\tilde{f}) \Big)
        \bigg) \\
        = \lim\limits_{b \uparrow a} \bigg( ((U \llcorner C_{\alpha}))((\gamma \circ f)d\omega) 
        - \partial (U \llcorner C_{\alpha}) (\gamma \circ f\omega) \\  
        + \int\limits_{0}^{1} dz \Big( (U \llcorner \{ f < a \}) \omega \wedge (\varphi_{z}'(\tilde{f})d\tilde{f}) \Big)
        \bigg) \\
        = - (U \llcorner C_{\alpha})(\gamma'(f)df \wedge \omega) + \int\limits_{0}^{1} dz \Big( (U \llcorner \{ f < a \}) \omega \wedge (\varphi_{z}'(\tilde{f})d\tilde{f}) \Big)
    \end{gather*}
    
\end{proof}

Let $f = \sqrt{x^2 + y^2 + z^2} : \overline{B}^{+} \rightarrow \mathbb{R}^{+}$ be a 1-Lipschitz function , and $C_{\alpha}$ be the cone as defined before. Using the definition of partial mass function from section \ref{section: formalisms}, we can define the partial mass function of a slice of minimal area enclosure $S$ of $\Sigma$ with respect to metric $g$ at level $a \in \mathbb{R}$ as 

\begin{equation} \nonumber
    m_{r,\alpha}(S,p) = |S \llcorner (\{ f < a \} \cap C_{\alpha})|_g.
\end{equation}

Further in this section, we will make the point $p$ and the surface $S$ implicit in the notation of the partial mass of a slice, and fix the surface $S$ to be the outermost minimal area enclosure $\ts$ of $\s$ with respect to $g' = u^4 g$. Moreover, we will label $r$ and $\alpha$ as usual parameters as opposed to their appearance in lower indices. Using the above result, we will prove an inequality relating the mass of a slice of the \textit{outermost minimal surface area} $\tilde{\Sigma}$ with the rate of change of the \textit{partial mass function} with respect to the metric g defined as: $m(r,\alpha) = |\tilde{\Sigma} \llcorner (\{ f < r \} \cap C_{\alpha})|_g$. Geometrically, this is equal to the area of the surface lying inside a unit ball, intersected with a cone with opening angle parameter $\alpha$, both having their origin at a point $p \in \ts \cap \s$.

\begin{lemma} \label{lemma: partial_mass_slice_rate_of_change}
    Let $\tilde{\Sigma}$ be the outermost minimal area enclosure of $\Sigma$ with respect to $g'$. Let $p \in \Sigma \cap \tilde{\Sigma}$. Then for sufficiently large values of $\alpha > 0$, 
    \begin{equation} \nonumber
        |\langle \ts, f, r, \alpha \rangle|_g \leq m_{r}'(r,\alpha)
    \end{equation}
\end{lemma}

\begin{proof}
    Using lemma \ref{lemma: slice algebra} for $f = \sqrt{x^2 + y^2 + z^2}$ we get 
    
    \begin{gather*}
        \partial (\ts \llcorner (\{f < r\} \cap \tilde{C}_{\alpha})) (\omega) - (\partial \ts) \llcorner (\{f < r\} \cap \tilde{C}_{\alpha})(\omega) = \\ = - (\ts \llcorner \tilde{C}_{\alpha})(\gamma'(f)df \wedge \omega) + \int\limits_{0}^{1} dz \Big( (\ts \llcorner \{ f < r \}) (\omega \wedge \varphi_{z}'(\tilde{f})d\tilde{f}) \Big) \numberthis \label{eq: slice_ineq}
    \end{gather*}
    
    where $\tilde{C}_{\alpha} = \int\limits_{0}^{1} (\varphi_z \circ \tilde{f})dz $. By construction
    
    \begin{equation} \nonumber
        \gamma(t) = 1, \; t < b, \; \gamma(t) = 0, \; t > a, \; 0 \geq \gamma'(t) \geq \frac{-(1 + \epsilon)}{b - a}, \; b < t < a 
    \end{equation}
    
    \begin{eqnarray} \nonumber
        \varphi_z (\zeta) = \begin{cases}
            0 & \zeta > \alpha z \\
            1 & \zeta < b'
        \end{cases} \\ \nonumber
        0 \geq \varphi_z ' (\zeta) \geq -\frac{1+\epsilon}{b' - \alpha z}, \;\; b' < \zeta < \alpha z
    \end{eqnarray}
    
    for fixed $\alpha, z,$ and $ b' < \alpha z$. \\
    
    The term that is extra in here as compared to the slicing lemma of \cite{simon} is the integral
    
    \begin{equation} \nonumber
        \int\limits_{0}^{1} dz  (\ts \llcorner \{ f < r \})( \omega \wedge \varphi_{z}'(\tilde{f})d\tilde{f})
    \end{equation}

    We estimate this term from above as follows:
    
    \begin{gather*} 
         \bigg| \int\limits_{0}^{1} dz (\ts \llcorner \{ f < r \}) (\omega \wedge \varphi_{z}'(\tilde{f})d\tilde{f}) \bigg| \leq \\ 
         \leq \int\limits_{0}^{1} dz \bigg | \sup\limits_{\ts} |D\tilde{f}| \left( \frac{1 + \epsilon}{\alpha z - b} \right) \big| (\ts \llcorner \{ f < r \}) \llcorner \{ b' < \tilde{f} = \sqrt{x^2 + y^2} < \alpha z \} \big| |\omega| \bigg | \\ 
         \leq \int\limits_{0}^{1} dz \bigg| \left( \frac{1+\epsilon}{\alpha z - b} \right) \big| (\ts \llcorner \{ f < r \}) \llcorner \{ b' < \tilde{f} = \sqrt{x^2 + y^2} < \alpha z \} \big| \bigg| \numberthis \label{eqn: massderivative_ineq}
    \end{gather*}
    
    Now, we take modulus on both sides, use identity $|a + b| \leq |a| + |b|$ on the right hand side and take limit $b' \uparrow \alpha z$ in \ref{eq: slice_ineq} (under the integral with respect to $z$). We take the limit so that the function $\varphi_z \circ \tilde{f}$ can tend to be the indicator of the slice of cone $C_{\alpha}$ at the level $z$, which can then be integrated over $z$ to generate the whole cone. Note that the term under the integral, stays well defined while we take the limit. The LHS of \ref{eq: slice_ineq} converges to partial mass of the slice $|\langle \ts, f, r, \alpha \rangle|_g$, and the first term on the right hand side is $ \leq m'_{r}(r,\alpha)$ by the \textit{slicing lemma} of \cite{simon}. We perform similar analysis on the additional second term on the right hand side of \ref{eq: slice_ineq}. Using \ref{eqn: massderivative_ineq} we have

    \begin{gather*}
        \int\limits_{0}^{1} dz \lim\limits_{b' \uparrow \alpha z} \bigg| \left( \frac{1+\epsilon}{\alpha z - b} \right) \big| (\ts \llcorner \{ f < r \}) \llcorner \{ b' < \tilde{f} = \sqrt{x^2 + y^2} < \alpha z \} \big| \bigg | = \\ 
        = (1+\epsilon) \int\limits_{0}^{1} \bigg| \big| (\ts \llcorner \{ f < r \}) \llcorner \{ b' < \tilde{f} = \sqrt{x^2 + y^2} < \alpha z \} \big|'_{\alpha z} dz \bigg| \\ 
        = (1+\epsilon) \int\limits_{0}^{1} \bigg| \big| (\ts \llcorner \{ f < r \}) \llcorner \{ b' < \tilde{f} = \sqrt{x^2 + y^2} < \alpha z \} \big|'_{\alpha} \frac{1}{z} dz \bigg| \\ 
        = 0 \text{ (for sufficiently large values of $\alpha$)}
    \end{gather*}
    
    Here, sufficiently large value of $\alpha$ means large enough so that  the cone $C_{\alpha}$ contains the intersection $\ts \cap \{ f < r\}$ inside it, which exists since $|\Sigma \cap \ts|_g = 0$ and $r \in (0,\rho]$ for sufficiently small $\rho$ (as selected in the definition of the diffeomorphism $\Phi: \overline{U} \rightarrow \overline{B}^{+}$). Note that in the case, where there is a tangent plane at the point $p$, the derivative with respect to $\alpha$ under the integral will be trivially zero, since the function $\big| (\ts \llcorner \{ f < r \}) \llcorner \{ b' < \tilde{f} = \sqrt{x^2 + y^2} < \alpha z \} \big|$ will be zero for all values of $\alpha$. In either case, this would mean that $\big| (\ts \llcorner \{ f < r \}) \llcorner \{ b' < \tilde{f} = \sqrt{x^2 + y^2} < \alpha z \} \big|'_{\alpha} = 0$, hence the integral will also be zero.

\end{proof}

The two cases, considered at the end of the previous lemma, are precisely what we need to distinguish, since different arguments are needed for each one. Therefore, we divide our analysis in the following two parts:

\begin{enumerate}
    \item \textbf{Case 1.} $m(r,\alpha) \equiv 0 $ for all $r \in (0,\rho]$ and $\alpha > 0$
    \item \textbf{Case 2.} $m(r,\alpha) \neq 0 $ for all $r \in (0,\rho]$ and some sufficiently large $\alpha > 0$
\end{enumerate}

The proof of Case 2 is a straightforward application of Lemma \ref{lemma: partial_mass_slice_rate_of_change} to Lemma 26 of \cite{jeff}.

\begin{lemma} \label{lemma: zeropardensity_mgeq0}
    (Case 2) Let $p \in \ts \cap \Sigma$ where $\ts$, the outermost minimal area enclosure of $\Sigma$ with respect to $g'$. Let the function 
    \begin{equation} \nonumber
        m(r,\alpha) = |\ts \llcorner (B(p,r) \cap C_{\alpha} )|_g
    \end{equation}
    be positive for some value of $\alpha > 0$ and for all $r \in (0,\rho]$. Then for some $\alpha > 0 $ 
    \begin{equation} \nonumber
        \Theta_{\alpha}(\ts,p) > 0  
    \end{equation}
\end{lemma}

\begin{proof}
    Define for $0 < r \leq \rho$ and $\alpha > 0$
    \begin{equation} \nonumber
        m(r, \alpha) = |\ts \llcorner (B(p,r) \cap C_{\alpha})|_g  
    \end{equation}
    By the statement of the Lemma this function is positive for all $r > 0$ and for some $\alpha > 0$. Moreover, this function is monotone by definition. Therefore, for almost all $r \in [0,\rho], \; m'(r)$ exists. Define
    \begin{equation} \nonumber
        \gamma_r := \partial (\ts \llcorner (B(p,r) \cap C_{\alpha}))
    \end{equation}
    Now, we take $\alpha$ to be sufficiently large, use Lemma 5 and take into account that $\partial \ts = 0$. Moreover,
    \begin{equation} \nonumber
        m'(r) \geq |\gamma_r|_g
    \end{equation}
    by the isoperimetric inequality for integrable n-currents on open sets, given in the \textbf{Appendix} of this paper. 
    The rest of the proof fis same as that of Lemma 26 in \cite{jeff}.
\end{proof}

Therefore, we are left with only the case where the partial mass function $m(r,\alpha) = |\ts \llcorner (B(p,r) \cap C_{\alpha})|_g = 0$ for all $r \in (0,\rho]$ and $\alpha > 0$. This raises the question - \textit{can there exist an outermost minimal area enclosure $\tilde{\Sigma}$ of $\Sigma$ with respect to some $g'$, such that $m(r,\alpha) = |\tilde{\Sigma} \llcorner (B(p,r) \cap C_{\alpha})|_g = 0$ for all $r \in (0,\rho]$ and $\alpha > 0$}? We prove that this indeed, is not possible.

\begin{lemma} \label{lemma: zeropardensity_meq0} (Case 1)
    Let $\ts$ be the outermost minimal area enclosure of $\Sigma$ with respect to $g'$ and $p \in \Sigma \cap \ts$. Then $m(r,\alpha) = |\tilde{\Sigma} \llcorner (B(p,r) \cap C_{\alpha})|_g$ is not trivially zero. 
\end{lemma}

Our idea of proving this is to use the coarea formulaand show that there exists some finite positive number $\sigma > 0$, such that $\ts \cap B(p,\sigma) = \Sigma \cap B(p,\sigma)$, which contradicts the fact that $|\ts \cap \Sigma|_g = 0$. 

\begin{proof}
    Assume the contrary, let

        \begin{equation} \nonumber
            |\ts \llcorner (B(p,r) \cap C_{\alpha})|_g \equiv 0
        \end{equation}

        for all $r \in (0,\rho]$ and $\alpha > 0$. \\ 

        In particular, $|\ts \llcorner (B(p,\rho) \cap C_{\alpha})|_g = 0$ for all $\alpha > 0$.\\

        Let us assume that 
        \begin{equation} \label{eqn: massin_euclideanmetric}
            |\Phi(\ts \cap (B(p,\rho) \cap C_{\alpha}))|_{\delta} = 0, \; \forall \: \alpha > 0
        \end{equation}
         
         Then,

        \begin{equation} \nonumber 
            \Phi(\ts \cap B(p,\rho)) = B_{\delta}(0,1) \cap \{ z=0 \} = \Phi(\Sigma \cap B(p,\rho))
        \end{equation}

        The first of these equations follow since otherwise there will exist some $\alpha > 0$ such that $C_{\alpha}$ will intersect $\Phi(\ts \cap B(p,\rho))$. The second follows by the definition of the diffeomorphism $\Phi$ constructed at the beginning of the paper. \\

        Bu, this implies that $\Phi^{-1}(\Phi(\ts \cap B_{\delta}(p,\rho))) = \Phi^{-1}(\Phi(\Sigma \cap B_{\delta}(p, \rho)))$, therefore leading to the contradiction of the fact that $|\ts \cap \Sigma|_g = 0$. Hence, we only need to prove equation \ref{eqn: massin_euclideanmetric}. \\ 

        Note that the diffeomorphism $\Phi: \overline{B}(p,\rho) \rightarrow \{ z \geq 0 \} \cap \overline{B}_{\delta}(0,1)$ is Lipschitz and induces the map $d\Phi_x: T_xM \rightarrow T_{\Phi(x)}\mathbb{R}^3 \simeq \mathbb{R}^3$. Then, denoting $J_{\Phi}(y) = \sqrt{det (d\Phi_y)^{*}(d\Phi_y)}$ and using the coarea formula

        \begin{gather*} 
            \int\limits_{\ts \cap (B(p,\rho) \cap C_{\alpha})} J_{\Phi} d\hm^{2}_g = \int\limits_{\Phi(\ts \cap (B(p,\rho) \cap C_{\alpha}))} \hm^{0}((\ts \cap (B(p,\rho) \cap C_{\alpha})) \cap \Phi^{-1}(y)) dA(y) \\ 
            = \int\limits_{\Phi(\ts \cap (B(p,\rho) \cap C_{\alpha}))} dA(y) \\ 
            = |\Phi(\ts \cap (B(p,\rho) \cap C_{\alpha}))|_{\delta}
        \end{gather*}

        Moreover 

        \begin{gather*}
            0 \leq \int\limits_{\ts \cap (B(p,\rho) \cap C_{\alpha})} J_{\Phi} d\hm^{2}_g \leq \max\limits_{x \in \ts \cap (B(p,\rho) \cap C_{\alpha})} ( |J_{\Phi}|) \int\limits_{\ts \cap (B(p,\rho) \cap C_{\alpha})} d\hm^{2}_g = 0\\ 
        \end{gather*}

       However, this implies that $|\Phi(\ts \cap (B(p,\rho) \cap C_{\alpha}))|_{\delta} = 0$ as well.
        
\end{proof}

Lemmas \ref{lemma: zeropardensity_meq0} and \ref{lemma: zeropardensity_mgeq0} rule out the possibility for an outermost minimal area enclosure to have zero lower partial density.

\section{Proof of Theorem \ref{theorem: main_theorem}} \label{section: main_theorem}

\begin{proof} 

    The existence of metric $g'$ such that $|\s - \ts_{g'}|_{g'} < \delta$ follows from the proof of Theorem \ref{theorem: maintheorem_jau} in \cite{jeff}. We present here our proof of the claim that $\ts$ is disjoint from $\s$ under the assumption that maximizers for $\alpha(A)$ have bounded boundary data.\\
    
    Using the proof of Theorem 29 from \cite{jeff} and Proposition \ref{prop:Prop_1}, there exists a path $g(t)$ in $\mathcal{H}(g)$ such that the area of all minimal area enclosures with positive lower partial density increases at a uniformly positive rate. During variation, two different cases can happen, and for each case we obtain contradictions as follows:

    \begin{enumerate}
    
        \item \textbf{Case 1: The area of some minimal area enclosures of $\Sigma$ with zero lower partial density does not increase} \\ 
        
        In this case, with respect to the new metric obtained after variation along the normalised path $\overline{g_t}$, the outermost minimal area enclosure has zero lower partial density (since the areas of all minimal area enclosures of $\s$ with positive lower partial density increased, all new minimal area enclosures are those with zero lower partial density), contradicting Lemma \ref{lemma: zeropardensity_meq0}.
        
        \item \textbf{Case 2: The area of all minimal area enclosures of $\Sigma$ with zero lower partial density increases with variation of $g'$ in path $\overline{g_t}$ } \\ 
        
        In this case, the areas of all minimal area enclosures increases uniformly and strictly in the normalised path $\overline{g_t}$ for sufficiently small values of $t>0$ (c.f. Theorem 29, Proposition 36 of \cite{jeff}). This will lead to the contradiction that $g'$ is a maximizer for $\alpha(A)$.
        
    \end{enumerate}
\end{proof}

\section{Discussion} \label{section: discussion}

In this paper, we presented the proof of the conformal conjecture under the assumption of boundedness of the harmonic function $u$ in the definition of the maximizer $g' = u^4g$ (theorem \ref{theorem: main_theorem}). This result was used in \cite{Bray:2009voj} to prove the following theorem regarding \textit{zero area singularities}.

\begin{theorem}

    Suppose $g$ is an asymptotically flat metric on $M \backslash \partial M$ of nonnegative scalar curvature such that all components of the boundary $\s = \partial M$ are ZAS. Assume that there exists a global harmonic resolution $(\overline{g}, \overline{\varphi})$ of $\s$. Assume that the conformal conjecture holds. Then

    \begin{equation*}
        m_{\text{ADM}} \geq m_{\text{ZAS}}
    \end{equation*}

    The equality holds if and only if $(M,g))$ is a Shwarzschild ZAS of mass $m < 0$.
    
\end{theorem}

We have not defined $m_{\text{ZAS}}$ and \textit{global harmonic resolution} in this paper, the reader is advised to refer to \cite{Bray:2009voj} for further details on this. The complete resolution of the conformal conjecture is still an open problem in differential geometry. We state a few possible ways to approach the complete resolution. One possible direction is to prove that Theorem \ref{theorem: main_theorem} without the requirement that the maximizer is bounded (originally conjectured in \cite{Bray:2009voj}).

\begin{conjecture} 
    Conjecture \ref{conjecture:conf_conj} is true without any additional conditions.
\end{conjecture}

Speculatively, a more promising direction is to prove the boundedness of the maximizer.

\begin{conjecture}
    The maximizer of theorem 3 is bounded below: $u \leq C$ for some constant $C > 0$.
\end{conjecture}

Proof of either of these approaches would lead to the complete proof of the conformal conjecture (conjecture \ref{conjecture:conf_conj}).

\section{Appendix}

\begin{namedprop}[Proposition A.1] (Isoperimetric inequality \cite{simon}, \cite{jeff}) For $n,k \geq 1$, let T be an integral n-current in $\mathbb{R}^{n+k}$ with compact support and zero boundary. Then there exists integral (n+1)-current R with compact support such that $\partial R = T$ and 

    \begin{equation} \nonumber 
        |R|^{\frac{n}{n+1}} \leq c|T|
    \end{equation}

for some constant c depending only on n and k. In particular, if k=1, then for all bounded, (n+1) Hausdorff measurable sets $\Omega \subset \mathbb{R}^{n+1}$

    \begin{equation} \nonumber 
        \mathcal{H}^{n+1}(\Omega)^{\frac{n}{n+1}} \leq c \mathcal{H}^n(\partial \Omega)
    \end{equation}
    
\end{namedprop}

\end{document}